\documentclass{amsart}

\usepackage[utf8]{inputenc}
\usepackage[T1]{fontenc}
\usepackage{amsmath}
\usepackage{amsfonts}
\usepackage{amssymb}
\usepackage{graphicx}
\usepackage{color}
\usepackage{selinput}
\usepackage{caption}
\usepackage{subcaption}
\usepackage{cancel}
\usepackage{enumitem}
\usepackage[paperwidth=7in, paperheight=10in, margin=0.8in]{geometry}
\usepackage[backref,colorlinks,linkcolor=red,anchorcolor=green,citecolor=blue]{hyperref}
\usepackage{todonotes}
\usepackage{stmaryrd}

\usepackage{dsfont}
\usepackage{tensor}
\usepackage{mathtools}

\usepackage{tikz}
\usetikzlibrary{calc}
\usepackage[arrow, matrix, curve]{xy}
\usepackage{pgfplots}
\pgfplotsset{compat=1.8}
\usepackage{mathrsfs}
\usetikzlibrary{arrows}
\definecolor{codegreen}{rgb}{0,0.6,0}
\definecolor{codegray}{rgb}{0.5,0.5,0.5}
\definecolor{codepurple}{rgb}{0.58,0,0.82}
\definecolor{yqqqqq}{rgb}{0.5019607843137255,0.,0.}
\definecolor{xdxdff}{rgb}{0.49019607843137253,0.49019607843137253,1.}

\usepackage{pgfplots}
\pgfplotsset{compat=newest}

\newcommand{\R}{\mathbb{R}}

\newcommand{\vp}{\varphi}

\newcommand{\pa}{\partial}

\DeclareMathOperator{\sgn}{sgn}

\newtheorem{theorem}{Theorem}
\newtheorem{lemma}[theorem]{Lemma}

\theoremstyle{definition}

\theoremstyle{remark}

\newtheorem{remark}[theorem]{Remark}

\def\supp{{\rm supp}}
\def\({\begin{eqnarray}}
\def\){\end{eqnarray}}
\def\[{\begin{eqnarray*}}
\def\]{\end{eqnarray*}}
\def\dt{\partial_t}
\def\ds{\partial_s}
\def\dx{\partial_x}

\author{M. Fischer, L. Kanzler, C. Schmeiser}
\address{M. Fischer: Universit\"at Wien, Fakult\"at f\"ur Mathematik,  Oskar-Morgenstern-Platz 1, 1090 Wien, Austria, \& Universität Graz - Institut für Mathematik und Wissenschaftliches Rechnen, Harrachgasse 21, 8010 Wien, Austria, {\tt michi.fischer@univie.ac.at}}
\address{L. Kanzler: CEREMADE, Universit\'e Paris Dauphine,
Place du Mar\'echal de Lattre de Tassigny, F-75775 Paris Cedex 16, France, {\tt laura.kanzler@dauphine.psl.eu}}
\address{C. Schmeiser: Universit\"at Wien, Fakult\"at f\"ur Mathematik,  Oskar-Morgenstern-Platz 1, 1090 Wien, Austria, {\tt christian.schmeiser@univie.ac.at}}

\title[Repulsive Particles]{One-dimensional short-range nearest-neighbor interaction and its nonlinear diffusion limit}
\date{}

\begin{document}
\maketitle


\section*{Abstract}
\noindent Repulsion between individuals within a finite radius is encountered in numerous applications, including cell exclusion, i.e. avoidance of overlapping cells, bird flocks, or microscopic pedestrian models. We define such individual based particle dynamics in one spatial dimension with minimal assumptions of the repulsion force $f$ as well as their external velocity $v$ and prove their characteristic properties. Moreover, we are able to perform a rigorous limit from the microscopic to the macroscopic scale, where we could recover the finite interaction radius as a density threshold. Specific choices for the repulsion force $f$ lead to well known nonlinear diffusion equations on the macroscopic scale, as e.g. the porous medium equation. At both scaling levels numerical simulations are presented and compared to underline the analytical results. 

\medskip
\noindent \textbf{Keywords:} agent-based models, repulsive force, cell-exclusion, nonlinear diffusion limit, porous medium equation

\smallskip
\noindent \textbf{Mathematics subject classification:} 82C22, 35R37, 35K55, 92C15

\section{Introduction}

We consider a chain of particles keeping their order along a straight line and interacting with their neighbors by distance dependent
repulsive forces, which vanish above an equilibrium distance. Their movement is further influenced by external forces in the form of a position dependent velocity. For a finite number of such particles the distance between the first and the last particle will remain finite for all time. Our goal is to derive a macroscopic continuum model sharing this property, i.e. for an initial particle density with bounded support there should be a finite upper bound for the length of the support at later times. Such a model, in the form of a nonlinear diffusion equation with drift, will be derived from the particle model by a continuum limit. The macroscopic model is diffusive since we choose a friction dominated (overdamped) microscopic model, being motivated by the  dynamics of bacterial colonies living in viscous environments.

Repulsive effects with a finite radius are often used in microscopic particle systems and their corresponding kinetic and macroscopic models. In flocking models, such as the Cucker-Smale- and the Vicsek-model, see~\cite{4200853, PhysRevLett.75.1226}, they appear as part of an interaction between attraction and repulsion. Examples occur in the modelling of collective behaviour within sheep-herds~\cite{a:sheep}, fish schools~\cite{Carrillo2010} and bird flocks~\cite{doi:10.1137/090757290}. In our setting only neighbouring particles interact if their distance does not exceed a given threshold, which causes the model to combine dynamics of metric interactions, i.e. interactions depending on the distance defined on the state-space, and topological interactions, i.e. interactions depending on the relative separation, which is given by the number of intermediate individuals. For various interacting agent systems, e.g. bird flock dynamics, topological interactions serve as a very realistic model approach, hence its recent interest~\cite{blanchet2017kinetic}, or see also~\cite{Haskovec2013FlockingDA}  in a kinetic context. 

Repulsion forces between individuals are furhter highly relevant in modelling size exclusion effects. On the microscopic level, this has been studied for pattern formation in bacterial colonies~\cite{Doumic2020APM, duvernoy2018asymmetric, you2018geometry}. In general, such modelling of size exclusion is a macroscopic alternative to models based on e.g. cellular automata, compare~\cite{schlake2011mathematical}, or microscopic asymmetric exclusion process, see~\cite{burger2016flow}, in which size exclusion has no influence on the diffusive term, typically a linear diffusive-term is derived. The modelled repulsion can also be seen as a  cut-off potential. On a microscopic level this is often used for better computational speed and e.g. steadier movement of pedestrians as in the optimal-step-model, see~\cite{seitz2012natural}.  A further investigation of size-exclusion effects between interacting particles, even grouped in species, can be found in~\cite{BJ1, BJ2}. Performing the many particle limit under usage of the matched asymptotic expansion technique leads to a system of nonlinear cross-diffusion equations for each species. Remarkable is the effect the size-exclusion has on the diffusivities on the macroscopic scale: while excluded-volume interactions between particles of the same species favour enhancement of the diffusion rates, its diminishment of is caused by particles of the other species.

This nonlinear diffusion term can be microscopically interpreted as particles trying to reach a desired density. Density thresholds similar to velocity thresholds are an important topic in pedestrian dynamics, see~\cite{adrian2018crowding, doi:10.1137/18M1215980}, gaining new focus in the context of social distancing~\cite{MayrKoester2021}, and cannot be calibrated in other macroscopic models, see~\cite{doi:10.1137/18M1215980}. Moreover, nonlinear diffusion equations are widely studied in the form of the well-known porous medium equation. Similar to this work, in~\cite{Oe} various versions of the porous medium equation have been derived in the many particle limit, starting from systems of ordinary differential equations modelling their interaction. While in~\cite{Oe} the macroscopic limit was achieved in a mean-field sense with varying interaction intensities, our approach relies on the definition of the discrete density, measuring the inverse of the difference of two neighboring particle positions. For an overview over the theory of porous medium equations we refer the reader to~\cite{Vazquez:2006aa, Vazquez}. While usual porous medium models consider diffusion without threshold, our approach can be seen as a generalisation of the nonlinear diffusion of a gas expanding in porous media, where the repulsion of its particles is limited by its force only having finite range (see Section \ref{s:macro}). 

Related limits from microscopic particle dynamics on the real line to macroscopic equations can be found in ~\cite{di2019deterministic, FR}. In~\cite{FR} the authors derive a nonlinear diffusive equation coupled with an aggregation term from the corresponding dynamics of the particle positions. On the one hand, it is very similar to our work in methods and assumptions on the nonlinearity. On the other hand, our assumptions of a purely external forcing coupled with the nonlinear interactions show a different aspect of this problem, together with a stronger convergence result.

On a macroscopic level nonlinear diffusion equations with degenerate diffusivity lead to moving boundary problems. It is closely related to the Stefan-problem, see classical books~\cite{meirmanov2011stefan, rubinshteuin1971stefan}.  Analytically exact solutions are still matter of current research, see~\cite{ceretani2018exact, font2018one} and also numerically challenges occur e.g. also in cancer research, see~\cite{a:cancer1, ihsan2020godunov}. On an agent-based level, cancer growth was investigated in~\cite{motsch2018short}. More recently, in ~\cite{LMP}, an individual-based mechanical model describing the dynamics of two contiguous cell populations with different characteristic behaviours was formulated and its formal continuum model was presented. The latter leading to the aforementioned free boundary problem with a nonlinear diffusive term, which exhibits travelling wave solutions.

This work is structured as follows. In Section \ref{s:micro} the microscopic model in the form of an ODE system
is formulated and its characteristic properties are derived. Section \ref{s:macro} contains a formal derivation of the macroscopic 
model and a discussion of its qualitative properties. This includes the derivation of a Eulerian formulation of the model,
which is originally written in terms of Lagrangian coordinates. In Section \ref{sec:lim} the macroscopic limit is carried out rigorously, providing also a existence of solutions for the continuum model. Some of the formal results are illustrated by numerical simulations in
Section \ref{sec:num}. Finally, we conclude this article with Section \ref{sec:CaO}.

\section{The microscopic model -- individual based dynamics}\label{s:micro}

Consider a chain of $N+1$ point particles with time dependent positions $x_i(t)\in\R$, $0\le i\le N$, such that
$$
    x_0(t)\le x_1(t) \le \cdots \le x_N(t) \,.
$$
Neighboring particles $i$ and $i+1$ interact by a distance dependent repulsive force \newline $F(x_{i+1}-x_i)$, written as $F(r) = F_0 f(r/R)$
in terms of the dimensionless function $f$, which satisfies
\(\label{f-ass}
   f:\, [0,\infty)\to[0,1] \text{ is Lipschitz and nonincreasing,} \quad \supp (f) = [0,1] \,,
\)
i.e. there is no interaction between neighbors further apart than the equilibrium distance $R>0$. Additionally, the particles are under the action of a position dependent external force, written as $F_{ext}(x) = F_0v(\frac{x}{NR})$, where $v$ fulfills
\begin{align}\label{v-ass}
 	v \in W^{2,\infty}(\R) \,.
\end{align}
Balancing these forces with friction against a nonmoving environment (with friction coefficient $\mu>0$) leads to the ODE system 
\begin{align}\label{eq:IBM-unscaled}
\begin{split}
\mu \dot{x}_0 &=  - F\left(x_1 - x_0\right) - F_{ext}(x_0)\,,\\
\mu\dot{x}_i &=  F\left(x_i-x_{i-1}\right) -  F\left(x_{i+1}-x_i\right) - F_{ext}(x_i) \,, \qquad 1\le i  \le N-1 \,, \\
\mu \dot{x}_N &=  F\left(x_N - x_{N-1}\right) - F_{ext}(x_N) \,. 
\end{split}
\end{align}
We introduce a nondimensionalization by
\[
   x \to NRx \,,\qquad t\to N^2 \frac{\mu R}{F_0} t \,.
\]
This is a diffusive macroscopic rescaling (by the factors $N$ and, respectively, $N^2$) of the natural microscopic scaling. The scaled
system reads
\begin{align}\label{eq:IBM}
\begin{split}
	 \dot{x}_0 &=  - Nf\left(N(x_1 - x_0)\right) + v(x_0) \,,\\
	\dot{x}_i &=  N\left( f\left(N(x_i-x_{i-1})\right) -  f\left(N(x_{i+1}-x_i)\right)\right) + v(x_i) \,, \qquad 1\le i  \le N-1 \,, \\
	 \dot{x}_N &=  N f\left(N(x_N - x_{N-1})\right) + v(x_N)\,.
\end{split}
\end{align}
We shall mostly work with a reformulation in terms of the new unknowns
\begin{align}\label{eq:defom}
\omega_i:= N(x_i-x_{i-1}), \quad 1\le i \le N \,,
\end{align}
satisfying
\begin{align}\label{eq:IBM_om}
\begin{split}
\dot{\omega}_1 &= N^2 \left[ 2f\left(\omega_1\right) - f\left(\omega_2\right)\right] +N\left(v(x_1)-v(x_0)\right)\,, \\
\dot{\omega}_i &= N^2 \left[2f\left(\omega_i\right) - f\left(\omega_{i-1}\right) -f \left(\omega_{i+1}\right)\right] +N\left(v(x_{i})-v(x_{i-1})\right) \,, \quad 2 \le  i \le N-1 \,, \\
\dot{\omega}_N & = N^2 \left[ 2f\left(\omega_N\right) - f\left(\omega_{N-1}\right)\right] +N\left(v(x_N)-v(x_{N-1})\right) \,,
\end{split}
\end{align}
which is coupled to
\begin{align}\label{eq:IBMom_pos}
	\dot{x}_0 = - N f(\omega_1) + v(x_0), \quad \quad x_i = x_0 + \frac{1}{N} \sum_{j=1}^i \omega_j\,.
\end{align}
This system will be considered subject to initial conditions
\(\label{IC-om}
  \omega_i(0) = \omega_{i,0} \ge 0 \,,\qquad 1\le i \le N \,, \quad x_0(0) = x_{0,0}\,.
\)
We start with stating global existence and boundedness of the solution.
 
\begin{theorem}\label{theo:exuniq}
Let $f$ and $v$ satisfy \eqref{f-ass} and \eqref{v-ass} respectively and let $0 \leq \omega_{i,0} < \infty$ for all $i \in 1, \dots, N$. Then there exists a unique global solution $(\omega_1,\ldots,\omega_N) \in C^{1,1}\left([0,\infty)\right)^N$ 
of \eqref{eq:IBM_om}, \eqref{IC-om}, which satisfies
\begin{align*}
\omega_{\min} e^{-\gamma t} \le \omega_i(t) \le \omega_{\max} e^{\gamma t} \,,\qquad t\ge 0 \,,\, 1\le i\le N \,,
\end{align*}
where we defined
 \begin{align*}
\omega_{\min}:=\min_{1\le j\le N} \omega_{j,0} \,, \quad  \omega_{\max} := \max \left\{1,\,\max_{1\le j\le N} \omega_{j,0} \right\} \,, \quad \gamma:=\|v'\|_{\infty}\,.
\end{align*}
\end{theorem}

\begin{proof}
Global existence and uniqueness follow from the Lipschitz continuity of $f$ and $v$. In order to show the upper bound of $\omega_i(t)$, we define $Y_i(t):=e^{-\gamma t}\omega_i(t)$ for $2 \leq i \leq N-1$ and estimate
\begin{align*}
	\dot{Y}_i(t) =& -\gamma Y_i(t) + e^{-\gamma t} N^2\left[ 2f\left(e^{\gamma t}Y_i(t)\right) - f\left(e^{\gamma t}Y_{i-1}(t)\right) -f \left(e^{\gamma t}Y_{i+1}(t)\right) \right] \\
	+& e^{-\gamma t}N \left(v(x_i)-v(x_{i-1})\right) \\
	\leq& e^{-\gamma t} N^2\left[ 2f\left(e^{\gamma t}Y_i(t)\right) - f\left(e^{\gamma t}Y_{i-1}(t)\right) -f \left(e^{\gamma t}Y_{i+1}(t)\right) \right] \,.
\end{align*}
Assume that for a time $t_0$ we have $Y_i(t_0)=\omega_{max}$. We remember $f \geq 0$ and $\supp(f) = [0,1]$, stated in \eqref{f-ass}. Hence, we have $f(\omega)=0$ for $\omega \geq \omega_{max}$ and we can conclude $\dot{Y}_i(t_0) \leq 0$, which gives the upper bound $\omega_i(t) \leq e^{\gamma t}\omega_{max}$. With analogous arguments we this upper bound for the boundary values $\omega_1$ and $\omega_N$ can be derived. 

In order to show the bound from below we define in a similar manner $y_i(t) := e^{\gamma t} \omega_i(t)$, from which we obtain the estimate
\begin{align*}
	\dot{y}_i(t) \geq N^2\left[ 2f\left(e^{-\gamma t}y_i(t)\right) - f\left(e^{-\gamma t}y_{i-1}(t)\right) -f \left(e^{-\gamma t}y_{i+1}(t)\right) \right]\,, \quad 2 \leq i \leq N-1\,.
\end{align*}
We assume that for a time $t_0$ we have $y_i(t_0)=\omega_{min}$ and $y_j(t_0)\geq \omega_{min}$ for $1 \leq j \leq N$. Due to the monotonicity of $f$ we see that $\dot{y}_i(t) \geq 0$ has to hold, which gives the lower bound $\omega_i(t) \geq \omega_{\min}e^{-\gamma t}$. In an analogous way the lower bound for the boundary points $\omega_1$ and $\omega_N$ can be shown, which completes the proof.
\end{proof}

\begin{theorem}\label{theo:char}
Let $(x_0,\ldots,x_N)$ be a solution of \eqref{eq:IBM}, satisfying \eqref{eq:defom}, \eqref{IC-om}. Then
\begin{enumerate}[label=\arabic*)]
\item \label{theo:char1} the center of mass
\begin{align*}
\bar{x} := \frac{1}{N+1}\sum_{i=0}^N x_i 
 \end{align*} 
 moves with respect to the average of all particle velocities:
 \begin{align}\label{eq:barx}
 	\frac{d}{dt}\bar{x}(t) = \frac{1}{N+1} \sum_{i=0}^N v(x_i(t)) \,,
 \end{align}
 \item the change in time of the variance
 \begin{align*}
 \mathcal{V}_x(t) := \frac{1}{N} \sum_{i=0}^N \left(x_i-\bar{x}\right)^2
 \end{align*}
 is given by
 \begin{align*}
 	\frac{d}{dt} \mathcal{V}_x(t) = \frac{2}{N} \left[\sum_{i=1}^N f(\omega_i) \omega_i  - \sum_{0 \leq i \leq N} (x_i-\bar{x})\sum_{1 \leq j \leq N, i \neq j} v(x_j)\right] \,,
 \end{align*}
\item the distance between the leftmost and the rightmost particle remains bounded in bounded time-intervals:
\[
   x_N(t) - x_0(t) \le e^{\gamma t} \omega_{\max} \,,\qquad t\ge 0\,.
\]
\end{enumerate}
\end{theorem}
\begin{proof}
The scalar product of \eqref{eq:IBM} with $(\vp_0(t),\ldots,\vp_N(t))$ gives, after summation by parts,
\begin{align}
\begin{aligned}\label{eq:weakx}
\sum_{i=0}^N \dot{x}_i \vp_i &= \sum_{i=1}^N f(\omega_i) N\left(\vp_i - \vp_{i-1}\right) +\sum_{i=0}^N \vp_i v(x_i) \,.
\end{aligned}
\end{align}
This immediately implies 1) with $\vp_0=\cdots=\vp_N = 1$. The choice $\vp_i(t) = \frac{2}{N}(x_i(t)-\bar{x})$, $i=0,\ldots,N$, gives the time evolution of $\mathcal{V}_x$ proving 2). Statement 3) is a consequence of 
 \[
   x_N(t) - x_0(t) = \frac{1}{N} \sum_{i=1}^N \omega_i(t) \,,
\]
and of the upper bound in Theorem \ref{theo:exuniq}.
\end{proof}

\begin{remark}
	We would like to point out that for the case where the particles are not exposed to an external force, i.e. $v \equiv 0$, the center of mass $\bar{x}$ is conserved and the variance is non-decreasing in time: 
	\begin{align*}
	\dot{\bar{x}}(t) = 0, \quad  \quad \dot{\mathcal{V}}_x(t) = \frac{2}{N}\sum_{i=1}^N f(\omega_i)\omega_i \geq 0\,, \quad \quad \text{for all } t \geq 0\,.
	\end{align*}
	Hence, the particle positions expand around their center of mass, but not too much.
\end{remark}

\section{The macroscopic model}\label{s:macro}

\subsection*{The continuum model in Lagrangian coordinates:} Interpreting the particle index as a discrete Lagrangian variable, the connection to the continuum is made by the definitions
$$
    \Delta s := \frac{1}{N} \,,\qquad s_i := i\Delta s \,,\quad 0\le i\le N \,,
$$
a discretization of the Lagrangian coordinate $s\in [0,1]$. Assuming the existence of a function $\omega(s,t)$, such that 
$\omega_i(t)\approx \omega(s_i,t)$ as $N\to\infty$, the formal limit
of the second equation in \eqref{eq:IBM_om} gives
\(\label{eq:mac-om}
   \dt\omega = - \ds^2 f(\omega) + \ds v(x) \,,\qquad 0<s<1 \,,
\)
a nonlinear diffusion equation with the diffusivity $-f'(\omega)\ge 0$, which is bounded by the Lipschitz continuity of $f$.
The limits of the first and third equation in \eqref{eq:IBM_om} lead to
\(\label{BC:mac-om}
   f(\omega(0,t)) = f(\omega(1,t)) = 0 \,,
\)
equivalent to 
\(
    \omega \ge 1 \,,\qquad s=0,1\,.
\)
This looks like incomplete information on the boundary. However, it is sufficient in view of the degenerate diffusivity. If, on the one
hand, $\omega > 1$ next to the boundary, then the diffusivity vanishes there, the solution does not change with $t$, and the boundary
condition is satisfied. If, on the other hand $\omega(0+,t)\le 1$ (or, respectively, $\omega(1-,t)\le 1$) then the solution needs to take
the boundary value 1.

Similarly, the continuum limit of the particle positions $x(s,t)$ in \eqref{eq:IBM}, satisfying $\ds x = \omega$, solves the Neumann type problem
\(
  &&\dt x = - \ds f(\ds x) + v(x) \,,\qquad 0<s<1 \,, \label{eq:mac-x}\\
  && \ds x\ge 1 \,,\qquad s=0,1 \,. \nonumber
\)
\subsection*{Eulerian coordinates -- the particle density:} Our next goal is to write the equation \eqref{eq:mac-om} in terms of the Eulerian coordinate $x$ instead of the Lagrangian
coordinate $s$. This produces a \emph{moving boundary problem} posed on the $x$-interval $[X_0(t),X_1(t)] := [x(0,t),x(1,t)]$.
The coordinate transformation can be written as
\begin{align}\label{eq:coordinationtransform}
x=X_0(t)+\int_0^s \omega (\sigma,t)d\sigma \,, 
\end{align}
implying
\begin{align*}
\ds \to \omega\dx,\quad  \dt \to \dt -\omega \dx f(\omega) \dx + v \dx \,,
\end{align*}
where \eqref{eq:mac-om} and \eqref{eq:mac-x} have been used. Consequently the Eulerian version of \eqref{eq:mac-om} 
reads
\begin{align*}
\dt\omega = -\omega^2 \dx^2 f(\omega) - v(x) \dx \omega  + v'(x) \omega \,,\qquad X_0 < x < X_1\,.
\end{align*}
This is equivalent to the conservation law 
\begin{align}\label{eq:mac-rho}
	 \dt\rho=\dx \left( \dx f\left(\frac{1}{\rho}\right) - v\rho \right)\,,
\end{align}
for the macroscopic particle density $\rho := 1/\omega$, which is complemented by the boundary conditions
$$
    \rho \le 1 \,,\qquad x=X_0,X_1 \,,
$$
and by the dynamics of the moving boundaries, determined from \eqref{eq:mac-x}:
\(\label{eq:x01}
    \dot X_{0,1} = \left[-\frac{1}{\rho} \dx f\left(\frac{1}{\rho}\right) + v \right] \Bigm|_{x=X_{0,1}} \,.
\)
\subsection*{Jump discontinuities:} 
As a consequence of the degeneracy of the diffusivity $D(\rho) := -\rho^{-2}f'(1/\rho)$ for $\rho\le 1$, the nonlinear diffusion equation 
\eqref{eq:mac-rho} supports jumps between values $\rho(x_*-,t) = 1$ and $\rho(x_*+,t)<1$. The velocity of the jump location $x_*(t)$
is given by the {\em Rankine-Hugoniot condition}
\(\label{eq:RH}
   \dot x_* = \frac{-\dx f(1/\rho)|_{x=x_*-}}{1-\rho|_{x=x_*+}} + v(x_*)\,.
\)
With the obvious changes also the case $\rho(x_*-,t) < 1$ and $\rho(x_*+,t)=1$ can be considered. Such jumps typically separate
regions where $\rho\ge 1$ from regions where $\rho<1$. In the case $\rho(X_{0,1}(t),t)=1$, the 
moving boundary equation \eqref{eq:x01} can be seen as a special case of \eqref{eq:RH}, with $\rho$ continued by zero outside
of $[X_0,X_1]$. Otherwise, when $\rho(X_{0,1}(t),t)<1$, the boundary does not move. Moreover, note that if $v \equiv 0$ and with the additional natural assumption $\dx \rho(x_*(t)-,t)\le 0$, the formula above implies $\dot x_*\ge 0$, i.e. the jump moves towards the region of lower density. 

\subsection*{Special cases for $v \equiv 0$}:
We below mentioned special cases will be considered in the case of no external force on the particles, i.e. $v \equiv 0$.  

First we note that the statements above do not cover the situation of initial data with a smooth transition between $\rho> 1$ and $\rho<1$. Consider $v \equiv 0$ and an initial datum $\rho_0(x)=1-cx$, $c>0$, and the choice $f(\omega) = (1-\omega)_+$. We expect that a discontinuity develops with location 
$x_*(t)$ starting at $x_*(0)=0$. Approximating the denominator on the right hand side of \eqref{eq:RH} by its value at $t=0$, we 
obtain
\begin{align*}
 \dot x_*\approx \frac{1}{x_*} \,,
\end{align*}
and therefore $x_\star(t) \approx \sqrt{2t}$ for small $t$. This behavior with infinite initial velocity also occurs in the Stefan-problem, see e.g.  \cite[Chapter 1, Example 1]{Andreucci2005LectureNO}. It can also be seen on a microscopic level in Figure \ref{fig:micro1a} for $x_0$ and $x_N$.

Another special case, again for $v \equiv 0$, is initial data with $\rho_0>1$ in a bounded interval, $\rho_0=1$ outside of it, and 
$f(\omega) = (1/\omega-1)_+^m$, $m \in \R$. For $m \ge 1$ the shifted density $\rho-1$ solves the {\em porous medium equation} with initial data with bounded 
support, a problem very well studied (see e.g. \cite{Vazquez}). In particular, in that case supp$(\rho-1)$ will grow, and the long-time behavior of 
the solution is given by an explicitly computable self-similar {\em Barenblatt profile.} For $m \in [0,1)$ we are in the case where $\rho-1$ solves the {\em fast diffusion equation}, which also is well investigated, see e.g. \cite{Carrillo:2003aa}. For $m<0$ this phenomenon is known as {\em super-fast diffussion}. For a survey of results regarding these types of nonlinear diffusion we refer the reader to \cite{Vazquez:2006aa}. If the initial data are such that the flux 
$\dx f(1/\rho)$ initially vanishes at the boundary of the support, then a {\em waiting time} phenomenon occurs, where the edges of the 
support start moving at a positive time. Indeed, for initial data where the right-most point is given by $x_*(0)$, such that $\rho(x_*(0)-)>1$ and $\rho(x_*(0)+)=1$, we can calculate under consideration $f\left(1/\rho\right) = (\rho-1)_+^m$ similar to \eqref{eq:RH} the velocity of the jump location
\begin{align*}
	\dot{x}_*(t) = m \Big(\rho(x_*(t)-,t)-1\Big)^{m-2} \pa_x\rho(x,t)_{|_{x=x_*(t)-}}\,.
\end{align*}
One can see clearly that the smaller the exponent $m$ and the closer the left-sided limit $\rho_0(x_*(0)-)$ is to the value 1, the flatter has to be the initial density $\rho_0$ left of the boundary point $x_*(0)$ in order to observe the aforementioned waiting time phenomenon.

Less clear is the situation with initial data of the form $\rho_0(x)=1-cx$, $c>0$ and a general nonlinearity $f$. We conjecture that,
on the one hand, a discontinuity develops with infinite initial speed as above, whenever $x = o(\dx f(1/\rho))$ as $x\to 0-$, and on
the other hand, the discontinuity only appears after a waiting time for $\dx f(1/\rho) = o(x)$ as $x\to 0-$. However, we are not aware
of any rigorous results on these questions.

\subsection*{Decay to equilibrium:}
The dynamics of \eqref{eq:mac-rho} dissipates the $L^2$-norm. After continuation of $\rho$ by zero outside of $[X_0,X_1]$, we
obtain
$$
   \frac{d}{dt} \int_{\R} \rho^2 dx = 2 \int_{\R} \frac{1}{\rho^2} f'\left(\frac{1}{\rho}\right) (\dx\rho)^2 dx - \int_{\R} \rho^2 v'(x) dx \,,
$$
which is non-positiv, if $v$ is non-decreasing. Moreover, if we again consider the special case $v \equiv 0$, we observe that the dissipation vanishes for $\rho< 1$ or $\rho$ independent from $x$. During the evolution we expect to see intervals $I_+(t)$, where 
$\rho\ge 1$, separated from intervals $I_-(t)$, where $\rho < 1$, by moving jump discontinuities, where $\rho=1$ are the boundary conditions for the intervals $I_+$. Therefore we expect that equilibria have intervals $I_{+,\infty}$, where $\rho\equiv 1$, separated by intervals $I_{-,\infty}$, where $\rho<1$ and otherwise arbitrary. The number of $I_{+,\infty}$-intervals might be smaller than that of
$I_+(0)$-intervals, since these intervals might merge by collisions of the moving jump discontinuities. As soon as this coarsening 
process is over, the limit $I_{+,\infty} = [a,b]$ of each interval $I_+(t)$ can be predicted. Let $[c,d]$ be big enough to contain $[a,b]$
with $\rho(c,t)=\rho_0(c),\rho(d,t)=\rho_0(d) < 1$. Then $a$ and $b$ can be computed from the conservation of mass and of the center of mass:
\begin{align}
   && \int_c^d \rho_0 dx = b-a + \int_{[c,d]\setminus [a,b]} \rho_0 dx \quad\Longrightarrow\quad b-a - \int_a^b \rho_0 dx = 0 \,, \label{cond:Iinfty1}\\ 
   && \int_c^d x\rho_0 dx = \frac{b^2-a^2}{2} + \int_{[c,d]\setminus [a,b]} x\rho_0 dx \quad\Longrightarrow\quad 
   \frac{b^2-a^2}{2} - \int_a^b x\rho_0 dx = 0 \,. \label{cond:Iinfty2}
\end{align}

\section{The rigorous macroscopic limit}\label{sec:lim}

The macroscopic limit will be carried out in the  individual based model in terms of the unknowns $\omega_i$, as in \eqref{eq:IBM_om}.
However, since the velocity of $x_0$ may be unbounded as $\Delta s = 1/N \to 0$ (see \eqref{eq:IBMom_pos}), we represent the variables $x_i$ in terms of
their average $\bar x$, satisfying \eqref{eq:barx}.
Therefore we consider the initial value problem
\begin{align}\label{eq:IBM_om_R}
\begin{split}
\dot{\omega}_1 &= \frac{2f\left(\omega_1\right) - f \left(\omega_2\right)}{\Delta s^2} + \frac{v(x_1)-v(x_0)}{\Delta s}\,, \\
\dot{\omega}_i &= - \frac{f\left(\omega_{i-1}\right)  - 2f\left(\omega_i\right) + f \left(\omega_{i+1}\right)}{\Delta s^2} + \frac{v(x_i)-v(x_{i-1})}{\Delta s}\,, \quad i=2,\ldots,N-1\\
\dot{\omega}_N &= \frac{2f\left(\omega_N\right) - f \left(\omega_{N-1}\right)}{\Delta s^2} + \frac{v(x_N)-v(x_{N-1})}{\Delta s}\,, \\
\omega_i(0) &= \omega_{i,0} \,, \qquad  i = 1,\ldots,N \,,
\end{split}
\end{align}
coupled to
\begin{align}\label{eq:IBM_xbar}
	\frac{d\bar x}{dt}  =  \frac{\Delta s}{1+\Delta s}\sum_{j=0}^N v(x_j)\,, \quad \bar x(0) = \frac{\Delta s}{1+\Delta s}\sum_{j=0}^N x_{j,0} \,,
\end{align}
with
\begin{align}\label{eq:IBM_xi}
	x_i = \bar x - \frac{\Delta s}{1+\Delta s}\sum_{j=1}^N (1 - (j-1)\Delta s)\omega_j + \Delta s \sum_{j=1}^i  \omega_j\,, \quad 0\le i \le N\,.
\end{align}
The last relation is obtained by averaging the expression for $x_i$ in \eqref{eq:IBMom_pos}, which allows to express $x_0$ in terms of $\bar x$.
For the initial data we still assume
\(\label{IC-ass}
   0 \le \omega_{\min} := \inf_{i\in\mathbb{Z}} \omega_{i,0} \,,\qquad \sup _{i\in\mathbb{Z}} \omega_{i,0}  =: \omega_{\max} < \infty \,.
\)
This problem is equivalent to \eqref{eq:IBM_om}, \eqref{eq:IBMom_pos}, and therefore the results of Theorem \ref{theo:exuniq} remain valid, i.e.
the existence and uniqueness of a global solution of \eqref{eq:IBM_om_R}--\eqref{eq:IBM_xi}, satisfying 
$$
	\omega_{\min} e^{-\gamma t} \le \omega_i(t) \le \omega_{\max} e^{\gamma t} \,,\quad t \ge 0 \,, i \in \mathbb{Z} \,,
$$
where we recall the notation $\gamma := \|v'\|_{\infty}$. The connection to the continuum is made by the definition of the piecewise constant interpolants
\begin{equation}\label{om-Delta}
   \omega_{\Delta s}(s,t) := \omega_i(t) \,,\, x_{\Delta s}(s,t) := x_i(t)\qquad\text{for } (i-1)\Delta s \le s < i\Delta s \,,\, t\ge 0 \,.
\end{equation}
With these definitions, \eqref{eq:IBM_xbar}, \eqref{eq:IBM_xi} imply
\begin{equation}\label{eq:xbar-x-approx}
\begin{split}
   \frac{d\bar x}{dt} &= \int_0^1 v(x_{\Delta s}(s,t))ds + O(\Delta s) \,,\quad \bar x(0) = \int_0^1 x_{\Delta s}(s,0)ds + \mathcal{O}(\Delta s) \,,\\
   x_{\Delta s}(s,t) &= \bar x(t) - \int_0^1 (1-s')\omega_{\Delta s}(s',t)ds' + \int_0^s \omega_{\Delta s}(s',t)ds' + \mathcal{O}(\Delta s) \,.
\end{split}
\end{equation}
For passing to the macroscopic limit in the nonlinearities $f$ and $v$, some regularity of $\omega$ and $x$ will be needed. Since the observations of the 
preceding section show that jump discontinuities of $\omega$ have to be expected, bounded variation is the best regularity we can hope for.

\begin{lemma}\label{lem:TV}
Let $(\omega_1,\ldots,\omega_N,\bar x, x_0,\ldots,x_N)$ be a solution of \eqref{eq:IBM_om_R}--\eqref{eq:IBM_xi}, with initial data satisfying \eqref{IC-ass} and 
$TV(\omega_{\Delta s}(\cdot,0))$ bounded idependently from $\Delta s$. Then 
$$
   \omega_{\Delta s},x_{\Delta s} \in L^\infty_{loc}([0,\infty);BV([0,1]))  \quad\mbox{uniformly as }\Delta s\to 0 \,,
$$
where $\omega_{\Delta s}$ is defined in \eqref{om-Delta}.
\end{lemma}

\begin{proof}
The result for $x_{\Delta s}(\cdot,t)$ is an immediate consequence of its monotonicity and of Theorem \ref{theo:char}, 3).
The bound on the total variation of $\omega_{\Delta s}$ will be obtained by calculating its time derivative:
\begin{eqnarray*}
	\frac{d}{dt} TV(\omega_{\Delta s}(\cdot,t)) &=& \frac{d}{dt} \sum_{i =1}^{N-1} |\omega_{i+1}-\omega_i | \\
	&=& - \sum_{i=1}^{N-1}\left( \sgn\left( \omega_{i+1}-\omega_i \right) - \sgn\left( \omega_{i}-\omega_{i-1} \right) \right) \frac{2f(\omega_i) - f(\omega_{i-1}) - f(\omega_{i+1})}{\Delta s^2} \\
	&& + \sum_{i =1}^{N-1} \sgn\left( \omega_{i+1}-\omega_i \right) \frac{v(x_{i+1}) - 2 v(x_i) + v(x_{i-1})}{\Delta s^2} \\
	&=& \Sigma_1 + \Sigma_2\,,
\end{eqnarray*}
where we used summation by parts for the formulation of the term $\Sigma_1$. By separately checking the cases that the signs are equal or different we observe 
$$
	\Sigma_1 \leq 0\,.
$$
For dealing with the second term $\Sigma_2$, we note that
\begin{align*}
	\frac{v(x_{i+1}) - 2 v(x_i) + v(x_{i-1})}{\Delta s^2} = v'(x_i)\left(\omega_{i+1}-\omega_i\right) + \frac{\Delta s}{2} \left( v''(\tilde{x}_i) \omega_{i+1}^2 + v''(\hat{x}_i) \omega_i^2\right)\,,
\end{align*}
where the remainder term contains mean values $\tilde{x}_i \in (x_{i}, x_{i+1})$ and $\hat{x}_i \in (x_{i-1}, x_{i})$. The boundedness of $\omega_{\Delta s}$ and Assumption 
\eqref{v-ass} imply
\begin{align*}
	\Sigma_2 \leq \gamma \,TV(\omega_{\Delta s}) + c(t)\,,
\end{align*}
for $c(t)>$ independent from $\Delta s$. Combining our results, we arrive at
\begin{align*}
	\frac{d}{dt} TV(\omega_{\Delta s}(\cdot, t)) \leq& \gamma\, TV(\omega_{\Delta s}(\cdot, t)) + c(t)\,,
\end{align*}
which gives a bound on $TV(\omega_{\Delta s}(\cdot,t))$ on every finite time interval by the Gronwall inequality. 
\end{proof}

These total variation bounds can be used to also get some regularity in time. 

\begin{lemma}\label{lem:simon}
With the assumptions of Lemma \ref{lem:TV}, $\dt \omega_{\Delta s}, \dt x_{\Delta s} \in L^\infty_{loc}([0,\infty); W^{-1,\infty}((0,1)))$ uniformly as $\Delta s \to 0$.
\end{lemma}

\begin{proof} For a test function $\vp\in W^{1,\infty}_0((0,1))$ we define
\begin{equation}\label{phi_i}
  \vp_i := \frac{1}{\Delta s} \int_{(i-1)\Delta s}^{i\Delta s} \vp\,ds \,,\qquad 
     J_i(t) := \frac{f(\omega_i(t)) - f(\omega_{i-1}(t))}{\Delta s} \,,
\end{equation}
and compute
\begin{align*}
&\int_0^1 \dt \omega_{\Delta s} \vp \, ds   = \sum_{i=1}^N \dot\omega_i \vp_i \Delta s \\
=& -\sum_{i=1}^N \frac{J_{i+1}-J_i}{\Delta s} \vp_i  \Delta s + \sum_{i=1}^N \frac{v(x_i)-v(x_{i-1})}{\Delta s}\vp_i  \Delta s \\
=& \sum_{i=1}^N J_{i} \frac{\vp_i-\vp_{i-1}}{\Delta s} \Delta s + \sum_{i=1}^N v'(\tilde{x}_i) \omega_i \vp_i  \Delta s\,,
\end{align*}
This leads to the estimate
\begin{align*}
   \left| \int_0^1 \dt \omega_{\Delta s}(s,t) \vp(s) \, ds \right| &\leq \|\vp'\|_{L^\infty((0,1))} \sum_{i=1}^N \left| f(\omega_i(t)) - 
       f(\omega_{i-1}(t)) \right| +  \gamma \|\vp\|_{L^\infty((0,1)) \sum_{i=1}^N \omega_i \Delta s} \\
   & \leq \|\vp'\|_{L^\infty((0,1))} L_f\,TV(\omega_{\Delta s}(\cdot,t))+  \gamma \|\vp\|_{L^\infty((0,1))}\omega_{max}e^{\gamma t} \,,\qquad t\ge 0 \,,
\end{align*}
where we have used the Lipschitz constant $L_f$ of $f$, Lemma \ref{lem:TV}, and Theorem \ref{theo:exuniq}.\\
For the time derivative of $x_{\Delta s}$ we proceed similarly:
\begin{align*}
& \left| \int_0^1 \dt x_{\Delta s} \vp \, ds \right| = \left| \sum_{i=1}^N \dot x_i \vp_i \Delta s \right| \\
\le & \sum_{i=1}^N f(\omega_i) \left| \frac{\vp_i-\vp_{i-1}}{\Delta s} \right| \Delta s + \sum_{i=1}^N \left| v(x_i)\vp_i\right|  \Delta s \\
=& \leq f(0) \|\vp'\|_{L^\infty((0,1))} +  v_{max} \|\vp\|_{L^\infty((0,1))} \,.
\end{align*}
\end{proof}

Since, for $1\le q<\infty$, $BV([0,1])\subset L^q((0,1)) \subset W^{-1,\infty}((0,1))$, where the first inclusion is compact \cite[Corollary 3.49]{AmbrosioFuscoPallara}, we conclude from Lemmas \ref{lem:TV} and \ref{lem:simon} and from \cite{simon1986compact} that $\{\omega_{\Delta s}, \Delta s>0\}$ and $\{x_{\Delta s}, \Delta s>0\}$ are relatively compact in $L^p_{loc}\left([0,\infty)\times [0,1]\right)$ for every $p < \infty$. This finally leads to the following convergence result:

\begin{theorem}\label{theo:limit}
Let $\omega_{\Delta s}(\cdot,0) \in BV([0,1])$ and let \eqref{IC-ass}, \eqref{f-ass}, and \eqref{v-ass} hold. Let $\omega_{\Delta s}$ be defined by \eqref{om-Delta} in terms of the 
solution of \eqref{eq:IBM_om_R}--\eqref{eq:IBM_xi}. Then $\lim_{\Delta s\to 0}\omega_{\Delta s} = \omega$ and $\lim_{\Delta s\to 0}x_{\Delta s} = x$ in $L^p_{loc}([0,\infty)\times [0,1])$ for any 
$1\le p<\infty$, restricting to appropriate subsequences. The limit $\omega \in L^\infty((0,\infty)\times (0,1))$ is a weak solution of the initial value problem for
\eqref{eq:mac-om}, satisfying the boundary conditions \eqref{BC:mac-om} weakly with respect to time. The function $x$ is determined from
\begin{align}
   \frac{d\bar x}{dt} &= \int_0^1 v(x)ds \,,\quad \bar x(0) = \int_0^1 x(s,0)ds  \,,\nonumber\\
   x(s,t) &= \bar x(t) - \int_0^1 (1-s')\omega(s',t)ds' + \int_0^s \omega(s',t)ds'  \,.\label{eq:xbar-x}
 \end{align}
\end{theorem}

\begin{proof}
For a test function $\vp\in C_0^\infty([0,\infty)\times (0,1))$ we test \eqref{eq:IBM_om_R} against $\vp_i(t)$ (defined in \eqref{phi_i}), noting that $\vp_1=\vp_N=0$
for $\Delta s$ small enough.
After an integration by parts with respect to $t\ge 0$ and summation by parts with respect to $i$ we obtain
\begin{align}
  &\int_0^1 \omega_{\Delta s}(s,0)\vp(s,0) ds + \int_0^\infty \int_0^1 \omega_{\Delta s} \dt\vp\,ds\,dt \nonumber\\
  =& \int_0^\infty \sum_{i=2}^{N-1} f(\omega_i) \frac{\vp_{i+1}-2\vp_i +\vp_{i-1}}{\Delta s^2} \Delta s\, dt 
  \,+ \int_0^\infty \sum_{i =2}^{N-1} v(x_i)\frac{\vp(x_{i+1})-\vp(x_i)}{\Delta s} \Delta s \, dt \label{om-weak-discr} \\
    =& \int_0^\infty \int_0^1 f(\omega_{\Delta s})\pa_s^2\vp \,ds\,dt \,+ \int_0^\infty \int_0^1 v(x_{\Delta s}) \pa_s\vp \, ds \, dt + \mathcal{O}(\Delta s)\,, \nonumber
\end{align}
where the last equation follows from
$$
    \frac{\vp_{i+1}-2\vp_i+\vp_{i-1}}{\Delta s^2} = \frac{1}{\Delta s} \int_{(i-1)\Delta s}^{i\Delta s} \pa_s^2\vp\,ds + \mathcal{O}(\Delta s) \,,
$$
and from
$$
	\frac{\vp_{i+1}-\vp_i}{\Delta s} = \frac{1}{\Delta s} \int_{(i-1)\Delta s}^{i\Delta s} \pa_s\vp \, ds +  \mathcal{O}(\Delta s)\,.
$$
Restricting to appropriate subsequences we have $\omega_{\Delta s}\to \omega$ in $L^p_{loc}\left([0,\infty)\times [0,1]\right)$ as
$\Delta s\to 0$, and we may pass to the limit in \eqref{om-weak-discr}:
\begin{align}
   \int_0^1 \omega_0(s)\vp(s,0)ds + \int_0^\infty \int_0^1 \omega \dt\vp\,ds\,dt  
   = \int_0^\infty \int_0^1 f(\omega)\pa_s^2\vp \,ds\,dt \,+ \int_0^\infty \int_0^1 v(x) \pa_s\vp \, ds \,,
\end{align}
where $\omega_0(s)$ is the limit of the initial data $\omega_{\Delta s}(0)$, which exists due to the assumption $\omega_{\Delta s}(0) \in BV(\R)$.
This is the weak formulation of the initial value problem for \eqref{eq:mac-om}.\\
Since $x_0$ and $x_N$ are bounded, we can pass to the limit in the weak formulations of 
$$
    f(\omega_1) = \Delta s (v(x_0) - \dot x_0) \,,\qquad f(\omega_N) = \Delta s (v(x_N) - \dot x_N)  \,,
$$
(see \eqref{eq:IBM}) to obtain the boundary conditions \eqref{BC:mac-om}.
Finally, \eqref{eq:xbar-x} is obtained by passing to the limit in \eqref{eq:xbar-x-approx}.
\end{proof}

\begin{remark}
For the case $v \equiv 0$ an existence theory for the continuous problem written in terms of $x(s,t)$ (see \eqref{eq:mac-x}) can also be carried out by interpreting 
the problem as gradient flow for the energy functional
$$
   E[x] := - \int_0^1 \int_0^{\ds x} f(p)dp\, ds \,.
$$
The basic theory (see e.g.  \cite{evans10}), however, only gives $x\in C([0,\infty); L^2((0,1)))$ and not much information on 
$\omega = \pa_s x$.
\end{remark}

\section{Numerical simulations}\label{sec:num}

\subsection{Microscopic model:}\label{subsec:numericasmicro} We illustrate the previous statements with numerical experiments in $x$ and $\omega$. We solve the systems \eqref{eq:IBM} and \eqref{eq:IBM_om} for the following two choices of nonlinearities
$$
	f_1(\cdot):=\left(1-\cdot\right)_+,\qquad f_2(\cdot):=\left(1-\cdot\right)_+^2
$$
with an implicit Euler algorithm to conserve the characteristic properties. We discretize as follows,
\begin{align}\label{eq:microdiscre}
\begin{aligned}
x_i^{n+1} =  & x_i^{n}+\frac{\Delta t}{\Delta s}\left[-f\left(\frac{x^{n+1}_{i+1}-x^{n+1}_i}{\Delta s}\right) + f\left(\frac{x^{n+1}_{i}-x^{n+1}_{i-1}}{\Delta s}\right)\right] \\ 
& +\Delta t \, v(x^{n+1}_{i}),\quad i \in \{1,\dots,N-1\},
\end{aligned}
\end{align}
including the boundary values $x_0$, $x_N$, based on classical ideas as in \cite{butcher2016numerical}. System \eqref{eq:IBM_om} can be calculated in every timestep from the results of \eqref{eq:microdiscre} to avoid coupling.

For $v\equiv0$ and $f=f_1$ we simulate $N=20$ agents and chose a time-stepping of $\Delta t=0.1\Delta s^2$, in which a typical parabolic CFL-condition in incorporated. The non-linearity in $f$ is solved with a fixed-point approach over $n=40$ iterations as proposed in \cite{kelley1995iterative}. The results can be found in Figure \ref{fig:micro1}.

\begin{figure}[htb]
 \centering
\begin{subfigure}[t]{0.45\textwidth}
\centering
\includegraphics[height=3.5cm]{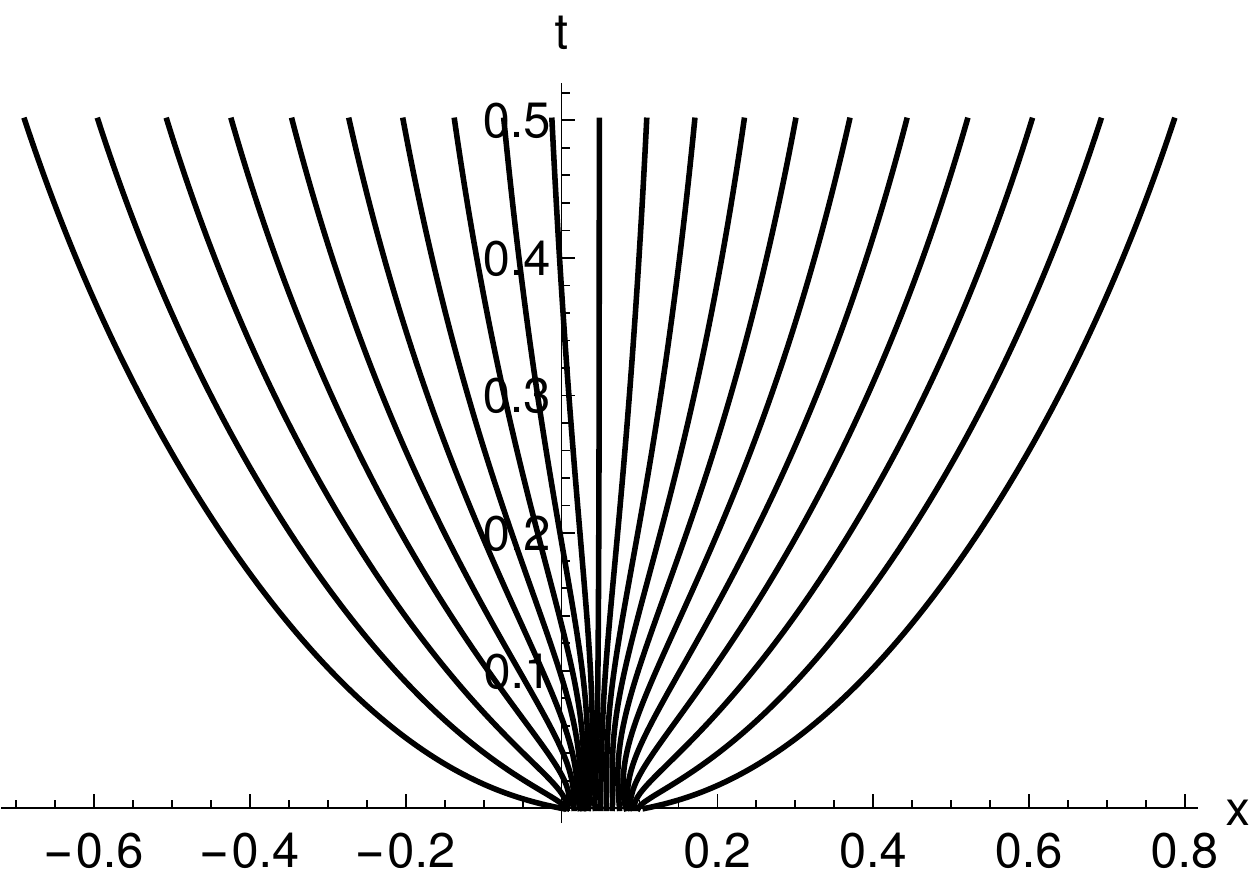}
 \caption{The trajectories $x_i$ over time where the initial values are chosen so that $x_i-x_{i+1}<\Delta s$ for all $i$. All agents are in interaction with their neighbours and the ensemble spreads.}
\label{fig:micro1a}
\end{subfigure}
\hspace{0.2cm}
\begin{subfigure}[t]{0.45\textwidth}
\centering
\includegraphics[height=3.5cm]{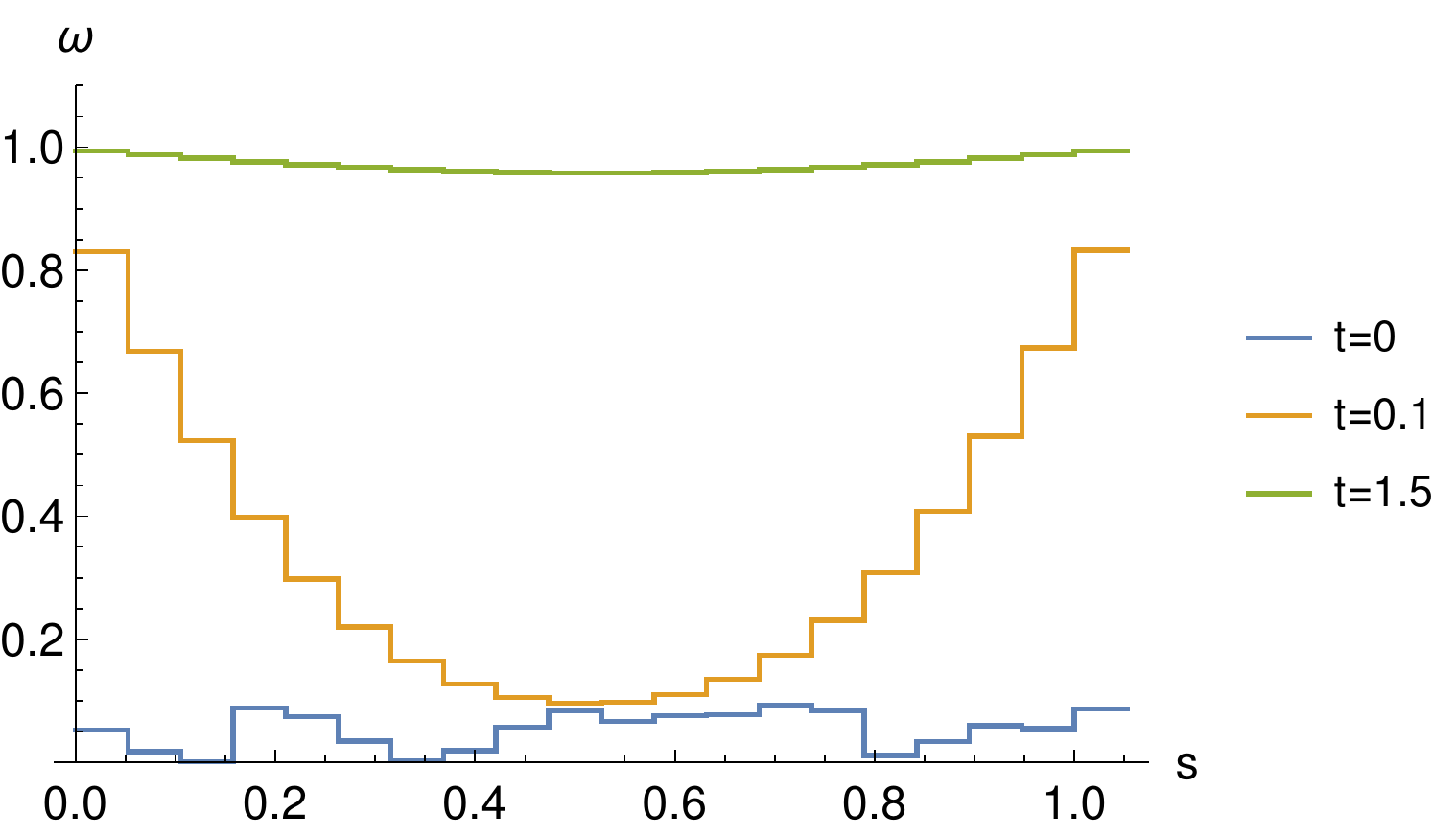}
\caption{The simulation associated with Figure \ref{fig:micro1a} in $\omega$, we see a smoothing effect with long term behavior $\omega_i \to 1$.} 
\label{fig:micro1b}
\end{subfigure}

\begin{subfigure}[t]{0.45\textwidth}
\centering
\includegraphics[height=3.5cm]{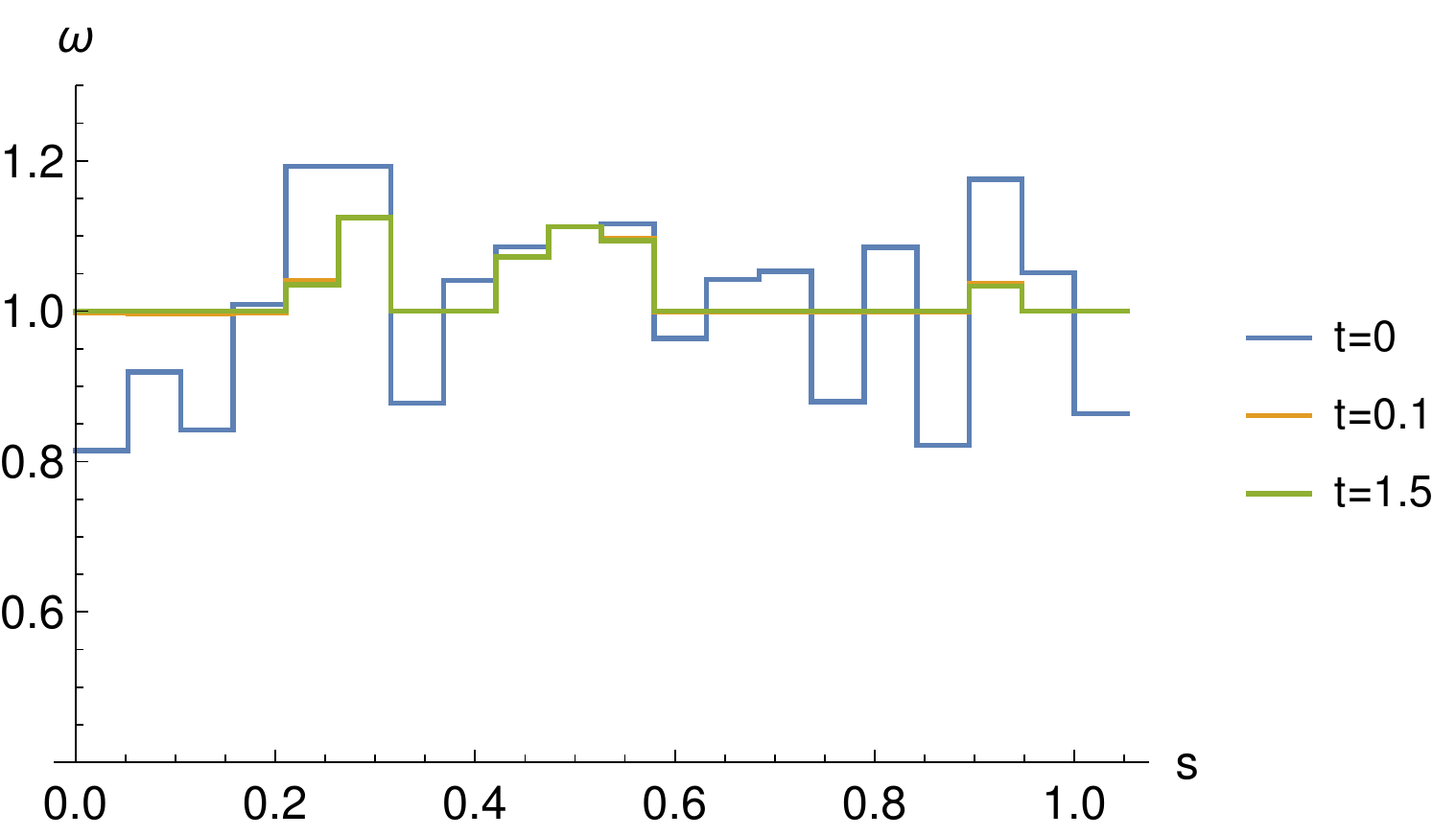}
\caption{Simulations in $\omega$. The smoothing effect stops before all $\omega_i$ reach the value 1. This behaviour is expected, since initially the threshold 1 is exceeded by a critical amount of $\omega_i$, hence some plateaus will not interact since the points in between already reached the value 1.} 
\label{fig:micro2b}
\end{subfigure}
\hspace{0.2cm}
\begin{subfigure}[t]{0.45\textwidth}
\centering
\includegraphics[height=3.5cm]{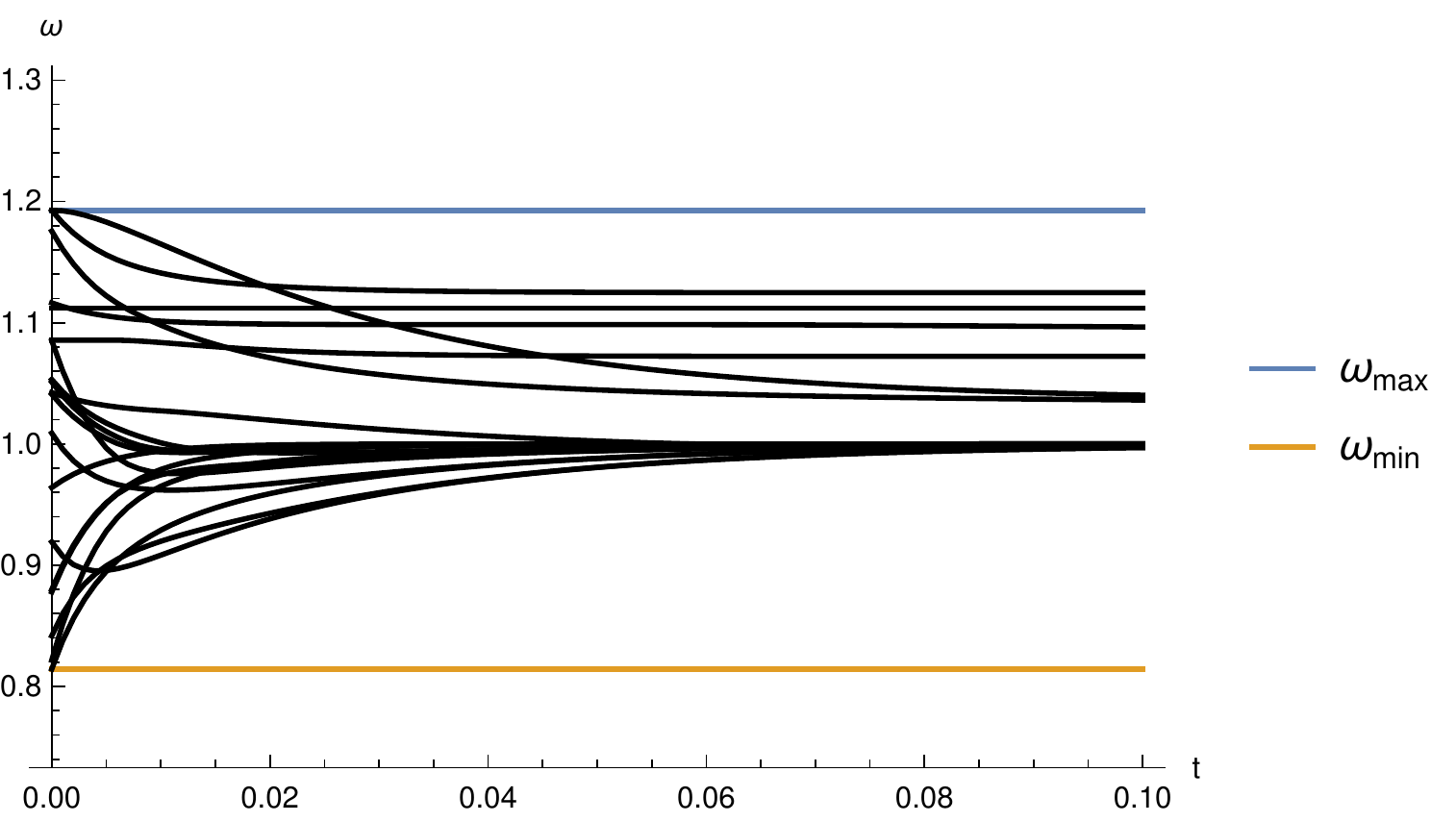}
\caption{The Min-Max-principle in Theorem \ref{theo:exuniq} visualized for the dynamics associated to Figure \ref{fig:micro2b}. Each black line corresponds to the time-evolution of a $\omega_i$.} 
\label{fig:micro3b}
\end{subfigure}
\caption{Time evolution of the discrete systems in $\omega$ and $x$ for $v\equiv 0$, $f=f_1$ and different initial values.}\label{fig:micro1}
\end{figure}

The Min-Max-principle in Theorem \ref{theo:exuniq}, visualized in Figure \ref{fig:micro3b}, already shows for $N=20$ the relation with a parabolic system which has in general a smoothing effect, as can be seen in Figure \ref{fig:micro1b}. However, this effect does not occur if the points $x_i$ are too far apart, which is visualized in Figure \ref{fig:micro2b}. Indeed, in that case groups of particles remain unmoved, since their metric distance is above the given interaction threshold. 

Figure \ref{fig:microdrift} corresponds to simulations for $f=f_1$ and non-trivial $v$, where the two simulated choices of $v$ can be seen in Figure \ref{fig:microD1a}. In Figure \ref{fig:microD1b} the agents are accelerating if their position is around $x=1$, caused by the enhanced value of $v_1$ . However, the distance (depending on $\Delta s$ and $N$) to the relative leading agent does not decrease due to the repulsion. For $v_{2}$ (orange in Figure \ref{fig:microD1a}) the agents congregate around position $x=1$, before the repulsion acts due to the resulting higher densities, which can be seen in Figure \ref{fig:microD1c}.

\begin{figure}[htb]
 \centering
\begin{subfigure}[t]{0.3\textwidth}
\centering
\includegraphics[height=2.75cm]{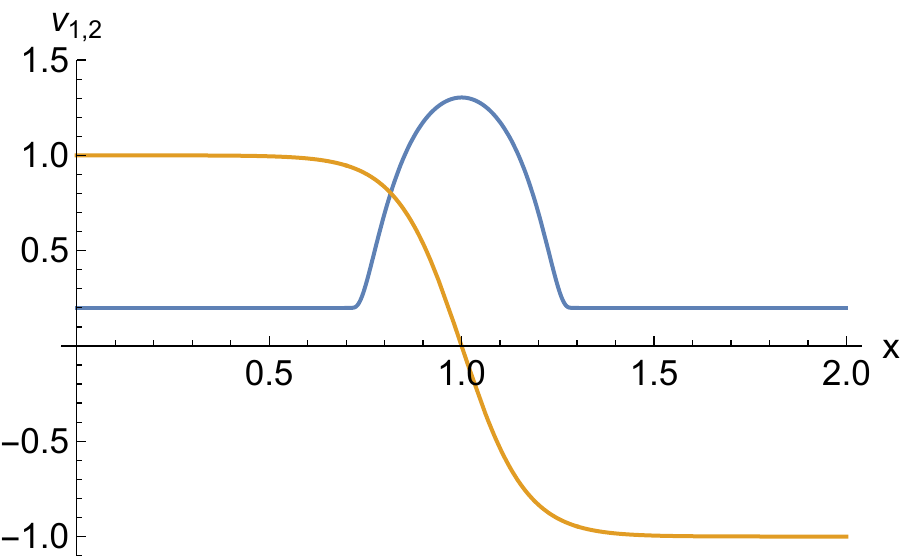}
 \caption{Velocities $v_{1}$ (blue) and $v_2$ (orange).}
\label{fig:microD1a}
\end{subfigure}
\hspace{0.2cm}
\begin{subfigure}[t]{0.3\textwidth}
\centering
\includegraphics[height=2.75cm]{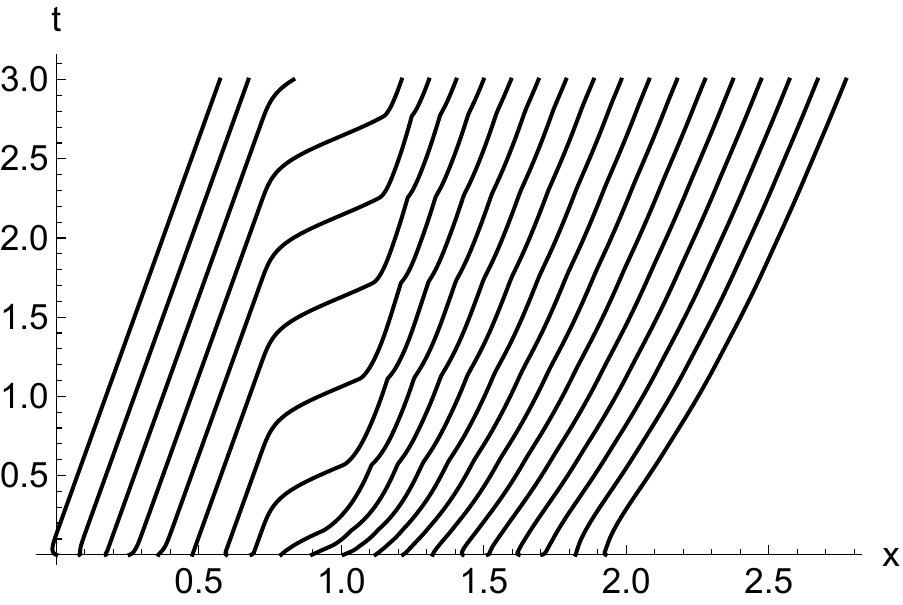}
\caption{The microscopic dynamics for $v_1$ (blue).} 
\label{fig:microD1b}
\end{subfigure}
\hspace{0.2cm}
\begin{subfigure}[t]{0.3\textwidth}
\centering
\includegraphics[height=2.75cm]{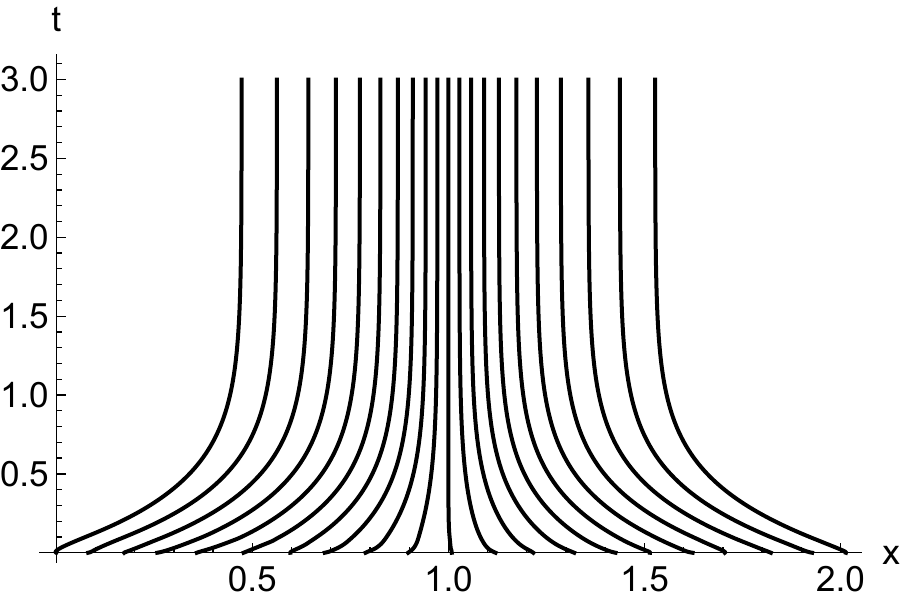}
\caption{The microscopic dynamics for $v_2$ (orange).} 
\label{fig:microD1c}
\end{subfigure}
\caption{Time evolution of the  discrete system in $x$ for $v\neq 0$.}\label{fig:microdrift}
\end{figure}

\subsection{Macroscopic model}\label{subsec:numericasmacro}
We investigate problem \eqref{eq:mac-rho} on an open domain $\mathbb{R}$ with $\rho_0>0$ and discretize it as $\rho (i\Delta x, j\Delta t)=\rho^j_i$ explicit in time via
\begin{align}\label{eq:macrodiscre}
\rho^{j+1}_i=\rho^{j}_i+\frac{\Delta t}{\Delta x^2} \left[ f\left(1/\rho^j_{i+1}\right) -2f\left(1/\rho^j_{i}\right) +f\left(1/\rho^j_{i-1}\right)\right] + \frac{\Delta t}{\Delta x} \left(\rho^j_{i+1}v^j_{i+1}-\rho^j_{i}v^j_{i}\right)
\end{align}
In the case $f=f_1$ we discretized the Laplacian by the usual finite difference approximation and used a first-order upwind scheme depending on the sign of $v$. 

We choose $\Delta x=0.001$ for a sharp visualisation of the shock and again $\frac{\Delta t}{\Delta x^2}=0.1$ in compliance with typical CFL-conditions, following classical literature, see~\cite{leveque1992numerical}. All initial-data chosen for the following simulations are typical sums of bump-functions of the form 
\begin{align}\label{eq:rho0macro}
	\rho_0(x):= \begin{cases}  h\exp\left({\frac{b^2}{(x-m)^2-b^2}}\right)\,, \quad &x \in [-b+m,\,b+m] \\ 0\,, \quad &x \notin (-b+m,\,b+m) \end{cases}
\end{align}
and constants. 

\begin{figure}[htb]
 \centering
 \begin{subfigure}[t]{0.45\textwidth}
 \centering
\includegraphics[height=4cm]{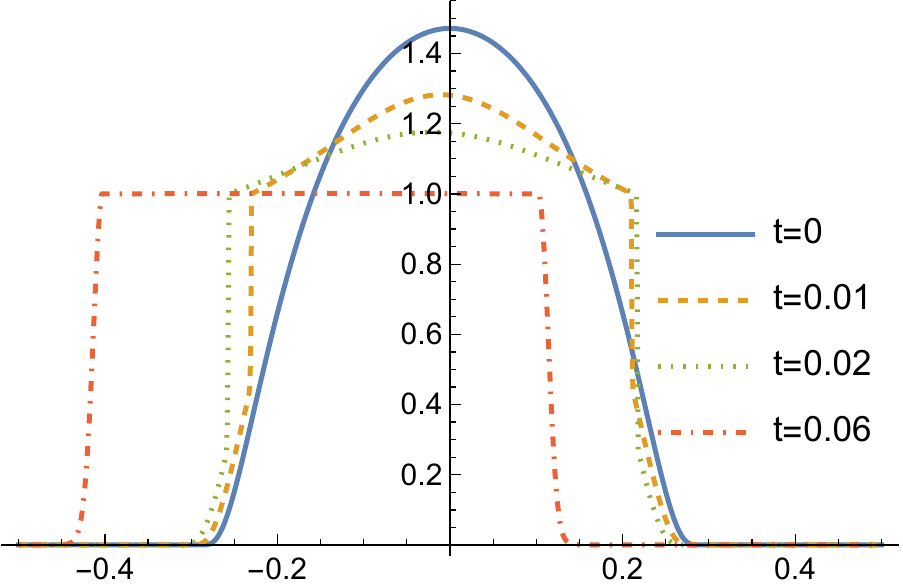}
\caption{Combination of drift with $v \equiv 1$ and diffusion with $f=f_1$ on the macroscopic scale.}
\label{fig:smoothingeffect}
\end{subfigure}
\hspace{0.2cm}
\centering
\begin{subfigure}[t]{0.45\textwidth}
\includegraphics[height=4cm]{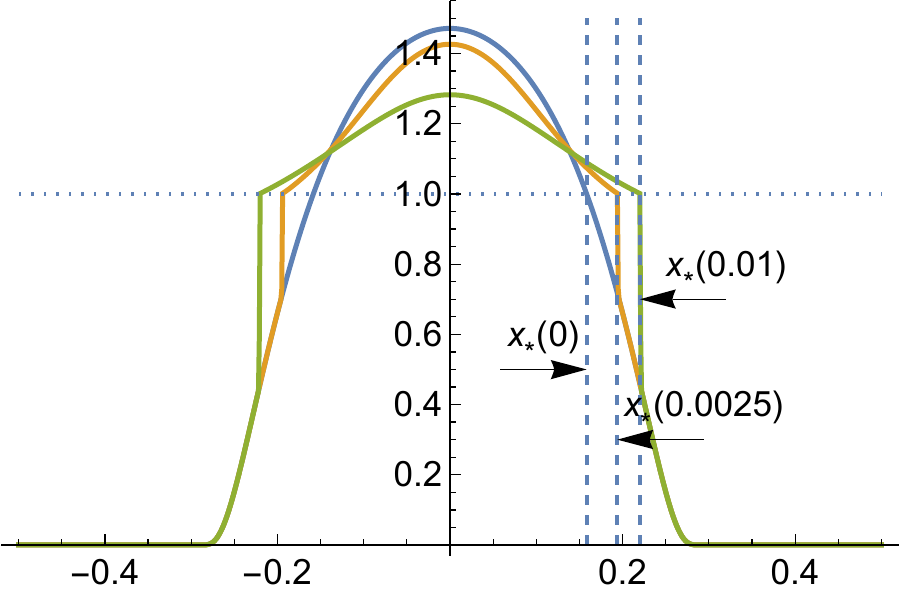}
 \caption{Smooth initial data and a shock evolving with time, which is created by $f=f_1$, for $v=0$. The position of the shock $x_\star$ at two different time-points was calculated from \eqref{eq:RH}.}
\label{fig:movingxs}
\end{subfigure}

\centering
 \begin{subfigure}[t]{0.45\textwidth}
\centering
\includegraphics[height=4cm]{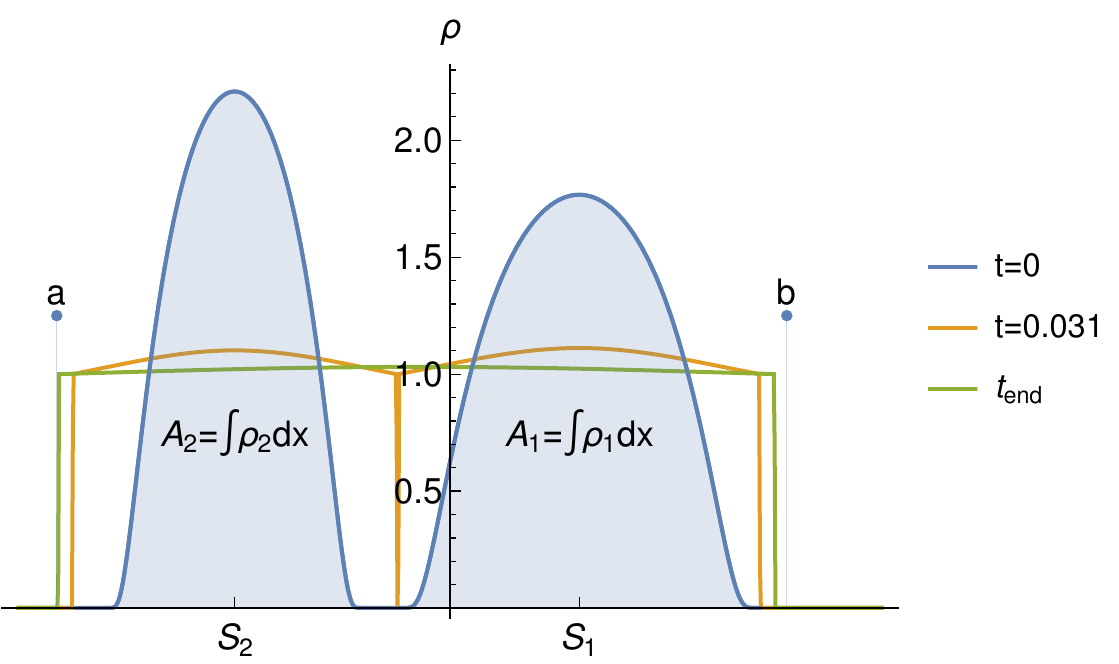}
 \caption{Two plateaus moving towards each other until they collide. The dynamics come to an end as soon the equilibrium size $[a,b]$ of the interval is reached.}
\label{fig:colliding}
\end{subfigure}
\hspace{0.2cm}
 \begin{subfigure}[t]{0.45\textwidth}
 \centering
\includegraphics[height=4cm]{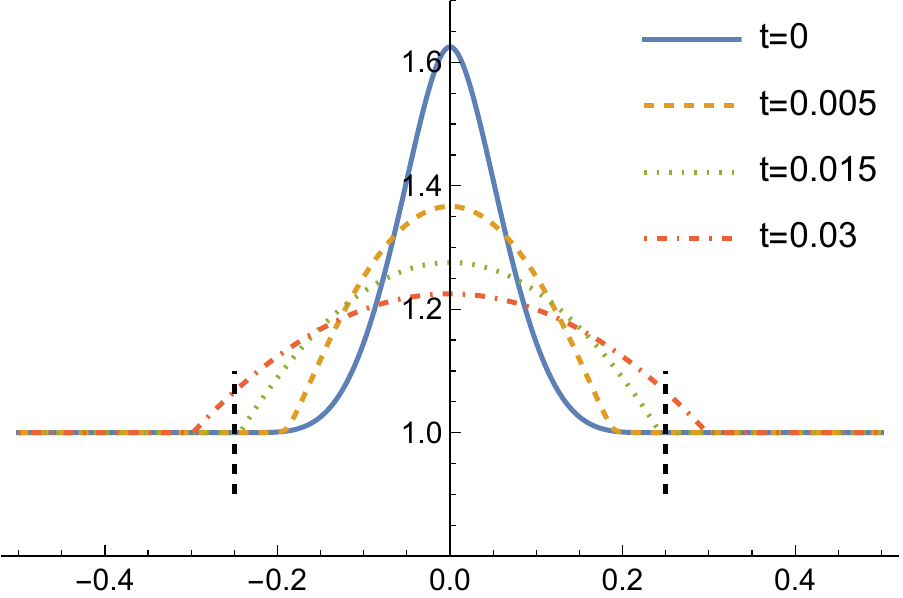}
\caption{Waiting-time phenomenon occurring for $f = f_2$ and initial data given by a fourth degree polynomial close to zero, which becomes flat outside. The points, where the initial data has the threshold-value 1 are marked with dashed lines.}
\label{fig:waiting}
\end{subfigure}
\caption{We visualize the combination of drift and diffusion as well as analytically described phenomena such as the shock velocity, the collision of two plateaus and its resulting steady state, and the waiting-time phenomenon.}\label{fig:rhomakro}
\end{figure}

In Figure \ref{fig:smoothingeffect} we see the combination of the diffusive effect and a drift to the left of system \eqref{eq:mac-rho} for the choices $f=f_1$ and $v\equiv1$. The diffusive effect acts stronger than the external forcing $v$.

Figure \ref{fig:movingxs} corresponds to simulations with $f=f_1$ and $v\equiv 0$ were the initial data were chosen such that they continuously pass through the threshold-value 1. The position of the jump $x_*$, which is formed for positive time, was calculated by \eqref{eq:RH}. We used for the discretization of \eqref{eq:RH} an explicit Euler algorithm for solving the ODE. Knowing $x_*$ is monotone  increasing and considering the left- and right-handed limits, we used in the numerator a downwind- and  for the denominator an upwind-approach.  We see that the discretisation \eqref{eq:macrodiscre} creates the correct shock-speed. Additionally continuous initial data lead to satisfying results since as mentioned, the discontinuity only occurs for $t=0$ and $x_*$ is Lipschitz for $t>0$. 

Figure \ref{fig:colliding} corresponds to the discussion regarding decay to equilibria for the case $v \equiv 0$ and $f=f_1$, which can be found at the end of Section \ref{s:macro}. We visualized a simulation for showing two colliding plateaus $\rho_{1,2}$ (with $\rho_0=\rho_1+\rho_2$), which are spreading in space, since each of their densities exceeds the threshold value 1, until they merge. The shocks move towards each other until the collision at time $t=0.031$ happens. The dynamics stop, once the final interval-length $[a,\,b]$ is reached, where the boundaries $a$ and $b$ are calculated by \eqref{cond:Iinfty1}-\eqref{cond:Iinfty2}.

In Figure \ref{fig:waiting} we visualized the in Section \ref{s:macro} mentioned waiting-time phenomenon, which occurs in porous medium equations. We chose $f=f_2$, no external forcing $v\equiv 0$ and initial data 
\begin{align*}
	\rho_0(x) := \begin{cases} -12800 | x| ^6+15360 | x| ^5-7200 | x| ^4+1600 | x| ^3-150 | x| ^2+\frac{13}{8} \,\, &\text{for }\, x \in  \left[-\frac{1}{4},\frac{1}{4}\right] \\ 1 \,\, &\text{else,} \end{cases}
\end{align*}
Therefore, $\partial_x f(1/\rho)$ initially vanishes at the boundary of the support of $\rho-1$, which is given by $x = \pm \frac{1}{4}$ and is marked with a dashed line.

\subsection{Comparison of the Microscopic and the Macroscopic Model}\label{subsec:numericasmicromacro}

We conclude the numerical simulations with an experiment, which compares the microscopic and the macroscopic dynamics in order to show consistency between the two scales. We start with a continuous initial density $\rho_0+0.5$ as in Equation \eqref{eq:rho0macro} (with $h=3$, $b=0.3$ and $m=0$).
We calculate the corresponding discrete values $x_i$ via \eqref{eq:coordinationtransform}. In Figure \ref{fig:vergleich} we see the solutions of the corresponding microscopic dynamics \eqref{eq:IBM_om}, solved numerically by \eqref{eq:macrodiscre}, and the one of the corresponding macroscopic dynamics \eqref{eq:mac-rho}, solved via \eqref{eq:microdiscre}, which are plotted beside each other. 
 
\begin{figure}[htb] \centering
 \begin{subfigure}[t]{0.45\textwidth}
 \centering
\includegraphics[height=3.5cm]{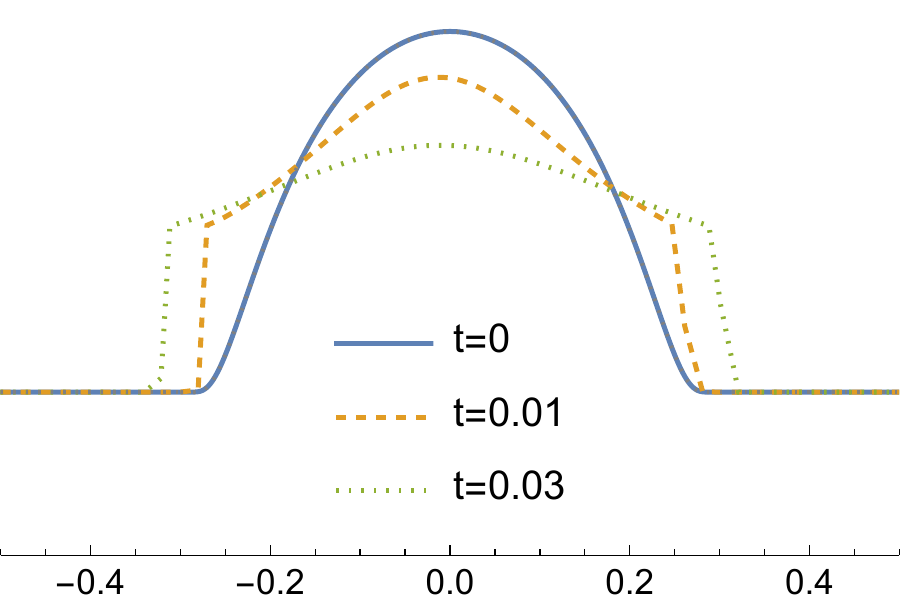}
\caption{Microscopic simulation for $N=40$ agents.} 
\end{subfigure}
\hspace{0.2cm}
\begin{subfigure}[t]{0.45\textwidth}
\centering
\includegraphics[height=3.5cm]{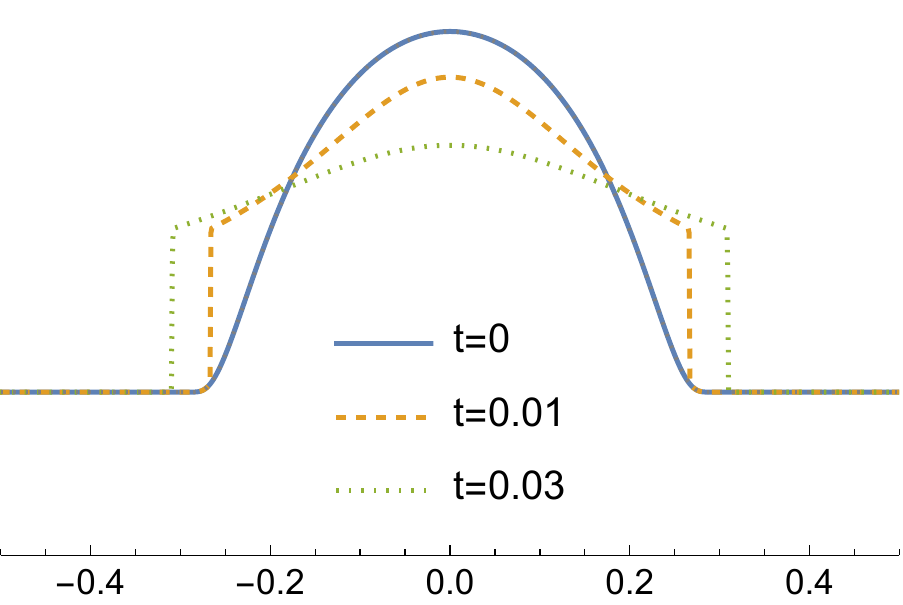}
\caption{Macroscopic simulation.}
\end{subfigure}
\caption{Comparison of the diffusive dynamics on both scales, we see a strong alignment already for $N=40$ agents.}
\label{fig:vergleich}
\end{figure}

\section{Conclusion}\label{sec:CaO}

\begin{figure}[htb]
 \centering
\includegraphics[height=3.5cm]{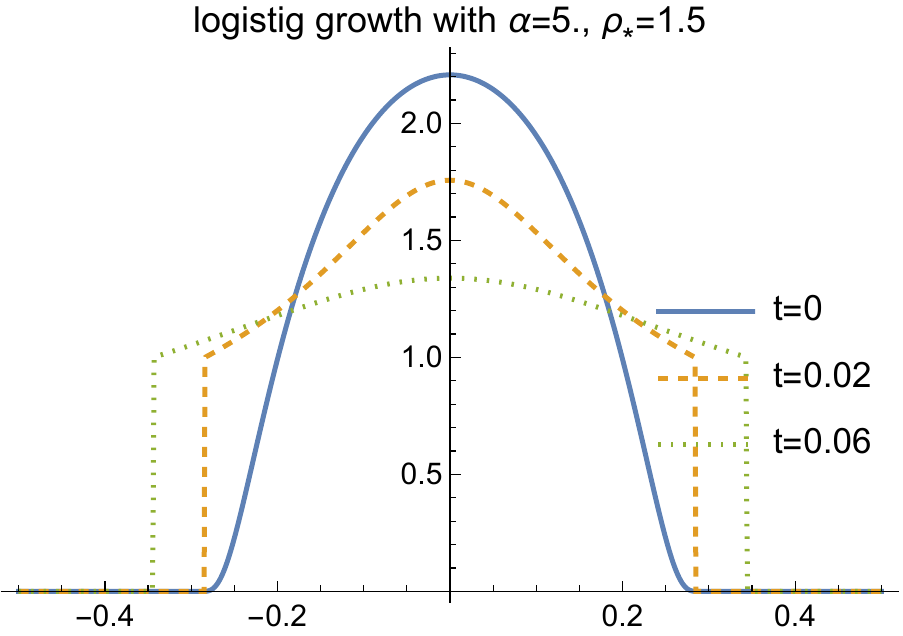}
 \caption{Simulation of equation \eqref{eq:growth}.}
\label{fig:log}
\end{figure}

In this work we modelled particles interacting in a certain radius and driven by an external force. We introduce the scaled distance $\omega$, which can also be seen as a derivative in Lagrangian coordinates $s$. From this interpretation, it was possible to define a microscopic density $\rho$ by inverting $\omega$.

For the microscopic systems \eqref{eq:IBM} and \eqref{eq:IBM_om} we could establish an existence and uniqueness result, together with a maximum-principle (Theorem \ref{theo:exuniq}). This result as well as observations of the particle's positions (Theorem \ref{theo:char}) are underlined with simulations in Section \ref{subsec:numericasmicro}. Properties of \eqref{eq:IBM_om} for $x$ and $\omega$ can be transferred to the density $\rho$.
 
On a macroscopic level our main focus was on \eqref{eq:mac-rho}, a conservation law for which we could derive jump conditions for discontinuous initial data, which can be related to the moving boundary of a Stefan problem. Moreover, we discussed specific choices for the repulsive force $f$, which lead to well known nonlinear diffusion equations as the porous medium equation or the fast diffusion equation. A rigorous limit from \eqref{eq:IBM_om} to \eqref{om-IVP} could be established in Theorem \ref{theo:limit}. Due to the non-linearity $f$, passing to the limit is non-trivial. In order to conclude by finding a weakly convergent sequence for system \eqref{eq:IBM_om}, we used bounds on the total variation with respect to the spatial variable of the solution and a compact interpolation theorem \cite{simon1986compact}. 

While this work has focused solely on diffusion in one spatial dimension, there are several natural extensions of the model that arise from practical considerations: for instance, source/sink terms, as well as aniosotropic effects that could arise in higher dimension. One motivating biological example that incorporate these effects is Bacteria growth:
\begin{align}\label{eq:growth}
		\pa_t \rho = \pa_x^2 f\left(\frac{1}{\rho}\right) + \alpha\rho(\rho_*-\rho),
\end{align}
	where $\alpha\rho(\rho_*-\rho)$ models the growth or death of bacteria depending on the carrying capacity $\rho_*>0$. In Figure \ref{fig:log} corresponding simulations for $f = f_1$, $\alpha = 0.5$ and $\rho_*=1.5$ can be found. 
	We leave the physical study of this and related models also in higher spatial dimensions as well as their rigorous derivation to future work.

\subsection*{Acknowledgments:} The authors are grateful to three anonymous referees, whose detailed comments led to a significant improvement of this article. Moreover, the authors thank \emph{Diane Peurichard} for fruitful discussions forming the starting point for these investigations.\\
 This work has been supported by the Austrian Science Fund, grants no. W1245 and F65. L.K. received funding by the European Commission under the Horizon2020 research and innovation programme, Marie Sklodowska-Curie grant agreement No 101034255.  \newline
\includegraphics[width=.1\linewidth]{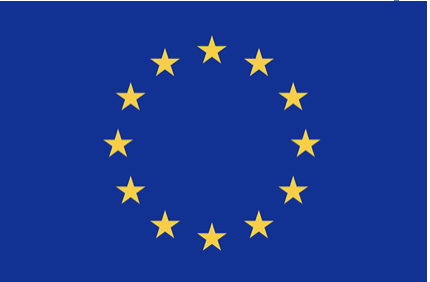}

\bibliographystyle{abbrv}
\bibliography{FKSlib.bib}

\end{document}